\documentclass[10pt,reqno]{amsart}
\usepackage[letterpaper,left=1.15in,right=1.15in,top=1in,bottom=1in,
  includeheadfoot,headheight=12pt,headsep=18pt,footskip=22pt]{geometry}
\usepackage{amsmath,amssymb,mathtools}
\usepackage{microtype}
\usepackage{graphicx}
\usepackage{enumitem}
\setlist[enumerate]{label=(\roman*)}
\usepackage{xurl}
\usepackage[pdfusetitle,hyperfootnotes=false,colorlinks=true,
  linktoc=page,linkcolor=blue,citecolor=blue,urlcolor=black]{hyperref}
\hypersetup{
  pdftitle={Hyperbolic area of Fatou components and new proofs of no-wandering-domain theorems},
  pdfsubject={Hyperbolic-area proofs for rational maps and transcendental entire functions},
  pdfkeywords={wandering domain, hyperbolic area, singular value, rational map, transcendental entire function}
}
\makeatletter
\def\ps@manuscript{%
  \ps@headings
  \def\@oddfoot{\normalfont\scriptsize\hfil\thepage\hfil}%
  \let\@evenfoot\@oddfoot
  \def\@oddhead{\setTrue{runhead}\normalfont\scriptsize\hfil
    \def\thanks{\protect\thanks@warning}\rightmark{}{}\hfil}%
  \def\@evenhead{\setTrue{runhead}\normalfont\scriptsize\hfil
    \def\thanks{\protect\thanks@warning}\leftmark{}{}\hfil}%
  \let\@mkboth\markboth}
\makeatother
\newcommand{\C}{\mathbb C}
\newcommand{\D}{\mathbb D}

\newcommand{\sphere}{\widehat{\mathbb C}}
\newcommand{\Fat}{\mathcal F}
\newcommand{\Jul}{\mathcal J}
\newcommand{\dd}{\,\mathrm d}

\theoremstyle{plain}
\newtheorem{theorem}{Theorem}[section]
\newtheorem{lemma}[theorem]{Lemma}
\newtheorem{corollary}[theorem]{Corollary}
\theoremstyle{definition}
\newtheorem{definition}[theorem]{Definition}
\newtheorem{example}[theorem]{Example}
\numberwithin{equation}{section}

\title[Hyperbolic area and no-wandering-domain theorems]{Hyperbolic area of Fatou components and new proofs of no-wandering-domain theorems}
\author{Zihao Ye}
\address{Zihao Ye: School of Mathematical Sciences, East China Normal University, Shanghai 200241, China}
\email{52290155037@stu.ecnu.edu.cn}
\date{September 30, 2026}

\begin{document}
\begin{abstract}
We give a new elementary proof of Sullivan's no-wandering-domain theorem for rational maps of degree at least two using hyperbolic area. We then develop this method further to prove that a transcendental entire function has no wandering Fatou components if its set of singular values is compact and all accumulation points of that set lie in the Fatou set. Neither proof uses quasiconformal techniques.
\end{abstract}
\maketitle
\tableofcontents

\section{Introduction}

Let $f$ be a rational map of degree at least two or a transcendental entire function. Its \emph{Fatou set} $\Fat(f)$ is the open set on which the iterates form a normal family in the spherical metric; its \emph{Julia set} $\Jul(f)$ is the complement in $\sphere$ or $\C$, respectively. A connected component of $\Fat(f)$ is called a \emph{Fatou component}. Such a component $U_0$ is \emph{wandering} if the components $U_n$ containing $f^n(U_0)$ are pairwise distinct, and \emph{preperiodic} otherwise. The early study of periodic Fatou components for rational maps goes back to the work of Fatou and Julia \cite{Fatou,Julia}. The possibility of wandering Fatou components for rational maps remained unresolved in their classical work \cite[pp.~403--404]{Sullivan}. More than sixty years later, Sullivan completely resolved this problem by proving his celebrated no-wandering-domain theorem, a central theorem of one-dimensional complex dynamics.

\begin{theorem}[Sullivan \cite{Sullivan}]\label{thm:rational}
Every Fatou component of a rational map $R:\sphere\to\sphere$ of degree at least two is preperiodic.
\end{theorem}

Sullivan's proof \cite{Sullivan} constructs quasiconformal deformations from a hypothetical wandering domain and obtains a contradiction from the finite-dimensionality of the space of rational maps of fixed degree. Zakeri's simplified presentation \cite{Zakeri}, following Baker and McMullen, uses Baker's reduction to eventually simply connected wandering components \cite[Lemma~1, pp.~652--653]{BKL}, which we also use below. Epstein's argument, presented for polynomials by Buff \cite[Section~2.4]{Buff}, works infinitesimally in the dual space of quadratic differentials: it avoids constructing quasiconformal deformations but still requires solving the linearized Beltrami equation $\bar\partial\xi=\mu$.

The search for a proof of Sullivan's no-wandering-domain theorem without quasiconformal methods has a long history. Bergweiler explicitly asked for proofs of the known no-wandering-domain theorems for rational maps and certain classes of meromorphic functions without using quasiconformal mappings \cite[Question~9]{Bergweiler}. For transcendental entire functions with infinitely many singular values, the classical deformation argument is generally unavailable \cite{MBRG}. In connection with the bounded-orbit wandering-domain problem, Bishop \cite[Section~2]{Bishop12} suggested seeking a new proof of Sullivan's theorem in which compactness of the orbit could replace finiteness of the singular set. In a 2014 MathOverflow discussion, Rempe \cite{Rempe2014} described finding a proof without quasiconformal methods or additional dynamical assumptions as ``a well-known open problem''.

Another motivation for seeking a new proof of Sullivan's theorem without quasiconformal techniques was to find new families of higher-dimensional polynomial maps without wandering Fatou components. Forn{\ae}ss and Sibony \cite[Question~2.2]{FS} suggested that finding a new proof in one variable might be a first step towards higher-dimensional results. Dujardin \cite[Question~3.2(2)]{Dujardin} posed the problem of finding substantial families of higher-dimensional rational maps without wandering domains. We hoped that a geometric argument in one dimension might provide tools for this problem.

We give a new elementary proof of Theorem~\ref{thm:rational} using hyperbolic area. The method combines two ideas. The first idea is to use area to obtain information beyond individual orbits, and the second is to approximate the geometric Fatou--Julia dichotomy by increasing finite unions of repelling cycles. The second idea is to understand how the Julia set separates the sphere into Fatou components, and to approximate dynamics within the Fatou set by following the iterates of Fatou points and components in the complements of these finite sets. We therefore first study the iterates of a Fatou component in a fixed punctured sphere, and then let the punctures become dense in the Julia set. Since subsequential limits on a hypothetical wandering component are constant maps into the Julia set, we use the hyperbolic metrics on the punctured spheres to reveal geometric information near Julia points. Combining this with the first idea, we study the hyperbolic area of the component measured in these punctured spheres, which converges to its intrinsic hyperbolic area as the punctures become dense. Figure~\ref{fig:punctured-spheres} illustrates this approximation.

\begin{figure}[htbp]
\centering
\includegraphics{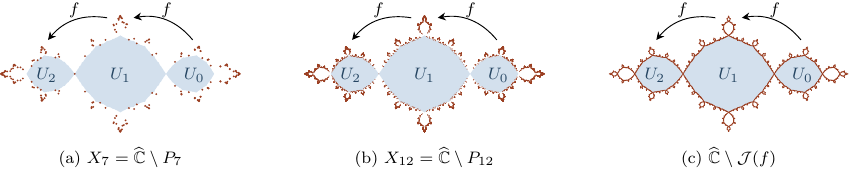}
\caption{Iteration of Fatou components for $f(z)=z^2-1$ in $X_n=\sphere\setminus P_n$. Here $P_n$ consists of all repelling periodic points of period at most $n$, and $\overline{\bigcup_n P_n}=\Jul(f)$.}
\label{fig:punctured-spheres}
\end{figure}

We then develop the area argument for transcendental entire functions. Let $S(f)$ be the closure in $\C$ of the finite critical and asymptotic values of $f$, and let $S(f)'$ denote its set of accumulation points. The Eremenko--Lyubich class $\mathcal B$ consists of transcendental entire functions with bounded $S(f)$; the Speiser class $\mathcal S\subset\mathcal B$ consists of those with finite $S(f)$. An area estimate for removing a compact subset of the Fatou set yields the following theorem.

\begin{theorem}[Main theorem]\label{thm:entire}
If $f\in\mathcal B$ and $S(f)'\subset\Fat(f)$, then $f$ has no wandering Fatou components.
\end{theorem}

The hypothesis allows infinitely many singular values and imposes no additional condition on their forward orbits. For $f\in\mathcal S$, the derived set $S(f)'$ is empty. We therefore obtain a proof without quasiconformal techniques of the no-wandering-domain theorem for $\mathcal S$, proved by Eremenko and Lyubich \cite{EL} and, independently, by Goldberg and Keen \cite{GoldbergKeen}.

\par\addvspace{.85\baselineskip}
\noindent\textbf{Further development.}\enspace
A long-standing problem in transcendental dynamics asks whether an entire function can have a wandering domain with bounded forward orbit \cite{EL84}, \cite[Problems~2.77 and~2.87]{BH89}, \cite[Question~1]{ELProblems}, \cite[Problems~2.77 and~2.87]{HL}. In \cite{Y}, the author develops the area method to settle this problem, proving that every point in a wandering domain of a transcendental meromorphic function has an unbounded forward orbit. The method is further developed in higher dimensions to completely resolve the local no-wandering-domain problem near Siegel invariant fibers of polynomial skew products with fiber maps of degree $d\ge2$. This extends the local no-wandering conclusion of Peters and Raissy \cite[Theorem~1]{PetersRaissy} to arbitrary linearizable elliptic base maps, without their restrictions on the central critical orbits. The same paper gives a complete affirmative answer to Bergweiler's Question~9 \cite[Section~4.6]{Bergweiler}; see \cite[Theorem~A and Section~4]{Y} for the proofs.

The area method introduced in this paper also forms the basis of the solutions to the bounded-orbit wandering-domain problem by Prochorov, Rempe and Waterman \cite{PRW}, and by Drach, Pardo-Sim\'on and U\v{c}akar \cite{DPU}. Their proofs of this result and the author's proof in \cite{Y} were obtained independently.

Section~\ref{sec:prelim} introduces the required preliminaries. Section~\ref{sec:rational} gives the new hyperbolic-area proof of Sullivan's theorem. Section~\ref{sec:entire} develops the argument further and proves Theorem~\ref{thm:entire}.

\section{Preliminaries}\label{sec:prelim}

We recall the basic notions and facts from hyperbolic geometry used below.

\begin{definition}\label{def:area}
An open set $G\subset\sphere$ is \emph{hyperbolic} if each component has $\D$ as its universal covering surface. For such a set $G$, the following notation is used:
\begin{enumerate}
\item The \emph{Poincar\'e metric} $\rho_G(z)|dz|$ is the complete conformal metric of curvature $-4$, normalized by $\rho_{\D}(z)=(1-|z|^2)^{-1}$;
\item The \emph{hyperbolic distance} is $d_G(z,w)=\inf_\gamma\int_\gamma\rho_G|dz|$, where the infimum is taken over piecewise smooth paths in $G$ joining $z$ to $w$, with $\inf\varnothing=\infty$;
\item The \emph{normalized hyperbolic area measure} is given in local coordinates by $\dd\alpha_G=(2/\pi)\rho_G^2\,\dd x\dd y$;
\item For a measurable set $A\subset G$, its area is $\alpha_G(A)=\int_A\dd\alpha_G$, and $\alpha_G(G)$ is the \emph{total intrinsic area} of $G$.
\end{enumerate}
\end{definition}

\indent\textbf{Facts.}
\begin{enumerate}[ref=(\roman*)]
\item\label{fact:punctured-area} For a finite set $F\subset\sphere$ with $|F|\ge3$, Gauss--Bonnet \cite[Theorem~10.4.3]{Beardon} and the Euler characteristic $2-|F|$ give $\alpha_{\sphere\setminus F}(\sphere\setminus F)=|F|-2$.
\item\label{fact:infinite-area} Every simply connected hyperbolic domain has infinite intrinsic area, since it is conformally equivalent to $\D$ and $\alpha_{\D}(\D)=4\int_0^1 r(1-r^2)^{-2}\dd r=\infty$.
\item\label{fact:metric-properties} For a holomorphic map $\psi:G\to H$ between hyperbolic open sets, the Schwarz--Pick theorem \cite[Theorem~2.11]{Milnor} gives $\rho_H(\psi(z))|\psi'(z)|\le\rho_G(z)$, equivalently $\psi^*\dd\alpha_H\le\dd\alpha_G$, with equality if $\psi$ is a covering. In particular, an inclusion $U\subset V$ gives \emph{domain monotonicity}: $\rho_U\ge\rho_V$ on $U$, $d_U(z,w)\ge d_V(z,w)$ for $z,w\in U$, and $\alpha_U(A)\ge\alpha_V(A)$ for every measurable $A\subset U$.
\end{enumerate}

We will approximate hyperbolic domains by finitely punctured spheres. The following lemma describes the convergence of their metrics.
\begin{lemma}\label{lem:recovery}
Let $F_1\subseteq F_2\subseteq\cdots\subset\sphere$ be finite sets with $|F_1|\ge3$. If $D$ is a component of $\sphere\setminus\overline{\bigcup_jF_j}$, then $\rho_{\sphere\setminus F_j}$ increases locally uniformly to $\rho_D$ on $D$.
\end{lemma}
\begin{proof}
Write $G_j=\sphere\setminus F_j$ and fix $a\in D$. Domain monotonicity in Fact~\ref{fact:metric-properties} gives $\rho_{G_j}(a)\nearrow\ell$ with $0<\ell\le\rho_D(a)$. Choose universal covering maps $\pi_j:\D\to G_j$ with $\pi_j(0)=a$. Covering invariance gives $|\pi_j'(0)|=1/\rho_{G_j}(a)$, with derivatives and densities taken in a local coordinate at $a$.

The maps $\pi_j$ omit the three or more points of $F_1$, so Montel's theorem gives a subsequence converging locally uniformly in the spherical metric to a limit $\pi$. Since $\pi(0)=a$ and $|\pi'(0)|=1/\ell>0$, the limit is nonconstant. Every point of $\bigcup_jF_j$ is omitted by all sufficiently late maps $\pi_j$. Hurwitz's theorem \cite{Ahlfors}, applied in local coordinates, shows that $\pi$ omits this union. As $\pi$ is open, its image also avoids the closure of the union. Hence $\pi(\D)\subset D$, and Schwarz--Pick (Fact~\ref{fact:metric-properties}) gives
\[
 \rho_D(a)\le\frac{1}{|\pi'(0)|}=\ell\le\rho_D(a).
\]
Thus $\rho_{G_j}(a)\to\rho_D(a)$ for every $a\in D$. The densities are continuous and increase pointwise, so Dini's theorem gives uniform convergence on every compact subset of $D$.
\end{proof}

The next estimate bounds the area added by finite punctures, even when the original domain has infinite area.
\begin{lemma}\label{lem:punctures}
If $G\subset\sphere$ is hyperbolic and $E\subset G$ is finite, then
\begin{equation}\label{eq:puncture-cost}
 0\le\int_{G\setminus E}(\dd\alpha_{G\setminus E}-\dd\alpha_G)\le |E|.
\end{equation}
\end{lemma}
\begin{proof}
Choose increasing finite sets $F_j\subset\sphere\setminus G$, containing a fixed triple, with union dense in $\sphere\setminus G$. Set $G_j=\sphere\setminus F_j$ and $H_j=G_j\setminus E$. Fact~\ref{fact:punctured-area} gives
\[
 \int_{H_j}(\dd\alpha_{H_j}-\dd\alpha_{G_j})=|E|.
\]
Lemma~\ref{lem:recovery} gives $\rho_{G_j}\to\rho_G$ and $\rho_{H_j}\to\rho_{G\setminus E}$ pointwise on $G\setminus E$. The area differences are nonnegative by Fact~\ref{fact:metric-properties}. Restricting the preceding integral to $G\setminus E$ and applying Fatou's lemma proves \eqref{eq:puncture-cost}.
\end{proof}

\section{The rational case}\label{sec:rational}

For a rational map $R$, let $\operatorname{CV}(R)\subset\sphere$ denote its set of critical values.

\begin{proof}[Proof of Theorem~\ref{thm:rational}]
Suppose that $R$ has a wandering Fatou component. By Baker's reduction \cite[Lemma~1, pp.~652--653]{BKL}, its forward images are simply connected for all sufficiently large iterates. Since $\operatorname{CV}(R)$ is finite, we may discard finitely many initial images and denote the remaining components by $(V_n)_{n\ge0}$, so that they are pairwise distinct, simply connected, and disjoint from $\operatorname{CV}(R)$. Each $V_n$ is a component of $R^{-1}(V_{n+1})$: complete invariance of the Fatou set gives one inclusion, and connectedness gives the other. Properness makes $R:V_n\to V_{n+1}$ a covering, and simple connectivity of $V_{n+1}$ makes it a conformal isomorphism. Hence
\[
 W=\bigsqcup_{n\ge0}V_n,
 \qquad R:W\longrightarrow W\setminus V_0
\]
is a conformal isomorphism.

By density of repelling periodic points \cite[Theorem~4]{Bergweiler}, choose increasing finite unions $P_j\subset\Jul(R)$ of complete repelling cycles such that $|P_1|\ge3$ and $\overline{\bigcup_jP_j}=\Jul(R)$. Put
\[
 X_j=\sphere\setminus P_j,\qquad
 Y_j=X_j\setminus\operatorname{CV}(R).
\]

Since $R(P_j)=P_j$, we have $R^{-1}(Y_j)\subset X_j$, and removing the critical values makes $R:R^{-1}(Y_j)\to Y_j$ an unramified covering. Moreover, $R(W)=W\setminus V_0\subset W\subset Y_j$, so $W\subset R^{-1}(Y_j)$. Fact~\ref{fact:metric-properties} therefore gives
\begin{equation}\label{eq:rational-pullback}
 \alpha_{X_j}(W)
 \le\alpha_{R^{-1}(Y_j)}(W)
 =\alpha_{Y_j}(W\setminus V_0).
\end{equation}

We now estimate the decrease in hyperbolic area from $W$ to $R(W)$ in the punctured sphere $X_j$ using \eqref{eq:rational-pullback}. Since $X_j$ and $Y_j$ have finite total area by Fact~\ref{fact:punctured-area},
\begin{align}\label{eq:rational-bound}
 \alpha_{X_j}(V_0)
 &=\alpha_{X_j}(W)-\alpha_{X_j}(W\setminus V_0)\notag\\
 &\le\alpha_{Y_j}(W\setminus V_0)-\alpha_{X_j}(W\setminus V_0)\notag\\
 &=\int_{W\setminus V_0}(\dd\alpha_{Y_j}-\dd\alpha_{X_j})
 \le\int_{Y_j}(\dd\alpha_{Y_j}-\dd\alpha_{X_j})\notag\\
 &=\alpha_{Y_j}(Y_j)-\alpha_{X_j}(X_j)
 =|\operatorname{CV}(R)\setminus P_j|
 \le 2\deg R-2.
\end{align}
Here the integrals are taken over the pointwise difference of the area densities, which is nonnegative by Fact~\ref{fact:metric-properties}; this justifies enlarging the domain of integration to $Y_j$. The last line follows from Fact~\ref{fact:punctured-area} and Riemann--Hurwitz.

Finally, Lemma~\ref{lem:recovery} gives $\rho_{X_j}\nearrow\rho_{V_0}$ on $V_0$. Monotone convergence and Fact~\ref{fact:infinite-area} give
\[
 \alpha_{X_j}(V_0)\nearrow\alpha_{V_0}(V_0)=\infty,
\]
contrary to \eqref{eq:rational-bound}.
\end{proof}

\section{\texorpdfstring{Functions in $\mathcal B$ with $S(f)'\subset\Fat(f)$}{Functions in class B with the derived singular set in the Fatou set}}\label{sec:entire}

The following lemma bounds the increase in hyperbolic area outside a fixed neighborhood of a removed compact set, independently of the number of punctures.

\begin{lemma}\label{lem:holes}
Let $K,N\subset\sphere$ be compact sets with $K\Subset\operatorname{Int}N$. There is a constant $C=C(K,N)<\infty$ such that, for every finite set $F\subset\sphere\setminus N$ with $|F|\ge3$, the open sets $G=\sphere\setminus F$ and $H=G\setminus K$ satisfy
\begin{equation}\label{eq:holes}
 \int_{G\setminus N}(\dd\alpha_H-\dd\alpha_G)\le C.
\end{equation}
\end{lemma}
\begin{proof}
If $K=\varnothing$ or $N=\sphere$, take $C=0$. Otherwise, fix a three-point set $Q\subset\sphere\setminus N$ and put $G_0=\sphere\setminus Q$.

First suppose that $Q\subset F$. Compact containment gives $\varepsilon=d_{G_0}(K,G_0\setminus\operatorname{Int}N)>0$. Since $G\subset G_0$, Fact~\ref{fact:metric-properties} implies $d_G(z,K)\ge\varepsilon$ for $z\in G\setminus N$. A universal cover $\phi:\D\to G$ with $\phi(0)=z$ therefore maps $\{|w|<r\}$ into $H$, where $r=\tanh\varepsilon$. Applying Fact~\ref{fact:metric-properties} to $w\mapsto\phi(rw)$ gives
\[
 r\rho_H(z)|\phi'(0)|\le1=\rho_G(z)|\phi'(0)|.
\]
Thus $u=\log(\rho_H/\rho_G)$ satisfies $0\le u\le\log(1/r)$ on $G\setminus N$. The curvature equation gives
\begin{equation}\label{eq:curvature}
 \Delta u=4(\rho_H^2-\rho_G^2)\ge0
 \quad\text{on }H.
\end{equation}
Here $\Delta=\partial_x^2+\partial_y^2$ in a local coordinate. Hence $u$ is subharmonic on $H$. The bound above holds near every point of $F$, since $F\cap N=\varnothing$. Defining its value at each such point by the upper limit gives a subharmonic extension to $\sphere\setminus K$.

Let $\sigma$ be the spherical metric and $L=\operatorname{Int}N\setminus K$. Choose $\chi\in C_c^\infty(\sphere\setminus K)$ with $0\le\chi\le1$, $\chi=1$ near $\sphere\setminus N$, and $\operatorname{supp}\Delta_\sigma\chi\Subset L$. On $L$, Fact~\ref{fact:metric-properties} and the inclusions $L\subset H$ and $G\subset G_0$ give $0\le u\le\log(\rho_L/\rho_{G_0})$. The subharmonic extension has a positive distributional Laplacian. Thus \eqref{eq:curvature} and integration by parts yield
\begin{align*}
 \int_{G\setminus N}(\dd\alpha_H-\dd\alpha_G)
 &\le\frac1{2\pi}\langle\Delta_\sigma u,\chi\rangle
 =\frac1{2\pi}\int_L u\Delta_\sigma\chi\,\dd A_\sigma\\
 &\le\frac1{2\pi}
  \int_L\log\frac{\rho_L}{\rho_{G_0}}\,|\Delta_\sigma\chi|\,\dd A_\sigma=:C_0.
\end{align*}
The integrand is supported in a compact subset of $L$, so $C_0$ is finite and independent of $F$.

For an arbitrary $F$, apply this bound to $F\cup Q$. Finite sets have zero area, so Fact~\ref{fact:metric-properties} and Lemma~\ref{lem:punctures} give
\begin{align*}
 \int_{G\setminus N}(\dd\alpha_H-\dd\alpha_G)
 &\le\int_{G\setminus(N\cup Q)}
       (\dd\alpha_{H\setminus Q}-\dd\alpha_{G\setminus Q})\\
 &\quad+\int_{G\setminus Q}
       (\dd\alpha_{G\setminus Q}-\dd\alpha_G)
 \le C_0+|Q\setminus F|\le C_0+3.
\end{align*}
Since $Q$ and $\chi$ were chosen using only $K,N$, the constant $C=C_0+3$ depends only on $K,N$.
\end{proof}

\begin{proof}[Proof of Theorem~\ref{thm:entire}]
Since $f\in\mathcal B$, every Fatou component is simply connected \cite[Proposition~3, p.~993]{EL}. As $S(f)'$ is a compact subset of $\Fat(f)$, choose compact neighborhoods $K,N\subset\Fat(f)$ satisfying
\begin{equation}\label{eq:KN}
 S(f)'\subset\operatorname{Int}K,
 \qquad K\Subset\operatorname{Int}N\Subset\Fat(f).
\end{equation}
If $S(f)'=\varnothing$, take $K=N=\varnothing$. The set $E=S(f)\setminus\operatorname{Int}K$ is finite: otherwise compactness of $S(f)$ would give an accumulation point outside $\operatorname{Int}K$.

Suppose that $f$ has a wandering Fatou component. Compactness of $N\subset\Fat(f)$ implies that $N$ meets only finitely many Fatou components; the same is true of the finite set $E$. We may therefore discard finitely many initial forward components and reindex the remaining ones as $(U_n)_{n\ge0}$, with $U_n\cap(N\cup E)=\varnothing$ for every $n$. Complete invariance of the Fatou set makes each $U_n$ a component of $f^{-1}(U_{n+1})$. Since $U_{n+1}$ avoids $S(f)\subset K\cup E$, the inverse-singularity covering theorem \cite[Section~3.2]{Sixsmith} makes $f:U_n\to U_{n+1}$ a covering. Simple connectivity of $U_{n+1}$ makes this restriction a conformal isomorphism. Consequently, with $W=\bigsqcup_{n\ge0}U_n$, the map $f:W\to W\setminus U_0$ is a conformal isomorphism.

By density of repelling periodic points \cite[Theorem~4]{Bergweiler}, choose increasing finite unions $P_j\subset\Jul(f)$ of complete repelling cycles such that $|P_1|\ge3$ and $\overline{\bigcup_jP_j}=\Jul(f)$. Put
\[
 X_j=\C\setminus P_j,\qquad Y_j=X_j\setminus S(f).
\]

Since $f(P_j)=P_j$, we have $f^{-1}(Y_j)\subset X_j$, and removing the singular values makes $f:f^{-1}(Y_j)\to Y_j$ a covering. Moreover, $f(W)=W\setminus U_0\subset W\subset Y_j$, so $W\subset f^{-1}(Y_j)$. Fact~\ref{fact:metric-properties} therefore gives
\begin{equation}\label{eq:entire-pullback}
 \alpha_{X_j}(W)
 \le\alpha_{f^{-1}(Y_j)}(W)
 =\alpha_{Y_j}(W\setminus U_0).
\end{equation}

We first bound the increase in hyperbolic area on $W$ caused by removing $S(f)$, uniformly in $j$. Since $S(f)\subset K\cup E$, we have $X_j\setminus(K\cup E)\subset Y_j\subset X_j$, and $W$ lies in all these sets.
\begin{align}\label{eq:entire-cost}
 \int_W(\dd\alpha_{Y_j}-\dd\alpha_{X_j})
 &\le\int_W(\dd\alpha_{X_j\setminus K}-\dd\alpha_{X_j})\notag\\
 &\quad+\int_W(\dd\alpha_{X_j\setminus(K\cup E)}
                   -\dd\alpha_{X_j\setminus K})\le C+|E|.
\end{align}
Here the integrals are taken over the pointwise differences of the area densities, which are nonnegative by Fact~\ref{fact:metric-properties}. For the first term on the right, apply Lemma~\ref{lem:holes} with $F=P_j\cup\{\infty\}\subset\sphere\setminus N$ and use $W\subset X_j\setminus N$. Applying Lemma~\ref{lem:punctures} to $X_j\setminus K$ with the finite puncture set $E\cap(X_j\setminus K)$ gives the bound for the second term.

We now estimate the decrease in hyperbolic area from $W$ to $f(W)$ in the punctured plane $X_j$. Since $X_j$ has finite total area by Fact~\ref{fact:punctured-area}, \eqref{eq:entire-pullback} and \eqref{eq:entire-cost} give
\begin{align}\label{eq:entire-bound}
 \alpha_{X_j}(U_0)
 &=\alpha_{X_j}(W)-\alpha_{X_j}(W\setminus U_0)\notag\\
 &\le\alpha_{Y_j}(W\setminus U_0)-\alpha_{X_j}(W\setminus U_0)\notag\\
 &=\int_{W\setminus U_0}(\dd\alpha_{Y_j}-\dd\alpha_{X_j})
 \le\int_W(\dd\alpha_{Y_j}-\dd\alpha_{X_j})
 \le C+|E|.
\end{align}

Finally, Lemma~\ref{lem:recovery} applied to $P_j\cup\{\infty\}$ gives $\rho_{X_j}\nearrow\rho_{U_0}$ on $U_0$. Monotone convergence and Fact~\ref{fact:infinite-area} give
\[
 \alpha_{X_j}(U_0)\nearrow\alpha_{U_0}(U_0)=\infty,
\]
contrary to \eqref{eq:entire-bound}.
\end{proof}

\begin{corollary}[Eremenko--Lyubich \cite{EL}; Goldberg--Keen \cite{GoldbergKeen}]
No function in $\mathcal S$ has a wandering Fatou component.
\end{corollary}
\begin{proof}
For $f\in\mathcal S$, the set $S(f)$ is finite, hence $S(f)'=\varnothing$. The conclusion follows from Theorem~\ref{thm:entire}.
\end{proof}

\begin{example}\label{ex:sine}
Let $\lambda=e^{2\pi i\theta}$, where $\theta$ is Diophantine irrational, and set
\[
 h_\lambda(z)=\lambda\frac{\sin^2z}{z},\qquad h_\lambda(0)=0.
\]
Then $h_\lambda\in\mathcal B\setminus\mathcal S$ and $S(h_\lambda)'=\{0\}$; see \cite[Example~4.25]{Y} for the singular-value calculation. Since $h_\lambda'(0)=\lambda$, the origin lies in a Siegel disk \cite{Siegel}, so Theorem~\ref{thm:entire} excludes wandering domains. This family is covered by neither the finite-type theorem nor the no-wandering criteria of \cite{ARS,MBRG,PS}.
\end{example}

\par\addvspace{.7\baselineskip}
\noindent\textbf{Acknowledgments.}\enspace
I thank Zhangchi Chen, Guizhen Cui and Songyan Xie for their helpful comments and suggestions on the manuscript.

\par\addvspace{.7\baselineskip}
\noindent\textbf{Disclosure of AI use.}\enspace
The author proposed using area and increasing collections of repelling periodic points to seek a proof of Sullivan's theorem without quasiconformal techniques. With the assistance of ChatGPT (OpenAI, GPT-6 astra), these ideas were combined into the punctured-sphere approximation and hyperbolic-area argument in Section~\ref{sec:rational}. The author then proposed extending the method to entire functions with infinitely many singular values. Under the author's direction, ChatGPT assisted in developing the uniform bound for the hyperbolic-area increase outside a fixed neighborhood of a removed compact set (Lemma~\ref{lem:holes}) and its application in Section~\ref{sec:entire}. The first complete drafts of the proofs of Theorems~\ref{thm:rational} and~\ref{thm:entire} and Lemma~\ref{lem:holes} were generated by ChatGPT, which also assisted with checking and revising the manuscript. The author personally verified the arguments, wrote the final proofs, and takes full responsibility for the mathematical content, references, and presentation.\par

\raggedbottom

\end{document}